\documentclass{amsproc}

\usepackage{tikz-cd}
\usepackage{enumitem}
\setlist[description]{leftmargin=\parindent,labelindent=\parindent}
\usetikzlibrary{matrix,arrows,decorations.pathmorphing}
\usepackage{amssymb}
\usepackage{amsmath,amscd}
\usepackage{mathrsfs}
\usepackage{url}
\usepackage{cite}
\usepackage{fullpage}
\usepackage{hyperref}
\usepackage{verbatim}
\usepackage{comment}
\usepackage{fancyvrb}
\usepackage{fvextra}
\usepackage{caption}

\newtheorem{thm}{Theorem}[section]
\newtheorem{prop}[thm]{Proposition}
\newtheorem{lem}[thm]{Lemma}

\newtheorem{conj}[thm]{Conjecture}

\theoremstyle{definition}
\newtheorem{definition}[thm]{Definition}

\numberwithin{equation}{section}

\renewcommand{\L}{\mathcal{L}}

\newcommand{\Q}{\mathcal{Q}}

\newcommand{\B}{\mathcal{B}}

\newcommand{\zz}{\mathbb{Z}}

\newcommand{\qq}{\mathbb{Q}}

\newcommand{\C}{\mathcal{C}}

\newcommand{\M}{\mathcal{M}}

\newcommand{\pp}{\mathbb{P}}

\renewcommand{\H}{\mathcal{H}}
\newcommand{\F}{\mathcal{F}}
\renewcommand{\P}{\mathcal{P}}

\renewcommand{\O}{\mathcal{O}}

\DeclareMathOperator{\rank}{rank}

\DeclareMathOperator{\ch}{ch}

\DeclareMathOperator{\GL}{GL}

\DeclareMathOperator{\coker}{coker}

\DeclareMathOperator{\PGL}{PGL}

\DeclareMathOperator{\BGL}{BGL}
\DeclareMathOperator{\codim}{codim}

\renewcommand{\hat}{\widehat}
\renewcommand{\ch}{\mathsf{A}}
\usepackage{wrapfig}
\renewcommand{\aa}{\mathbb{A}}

\newcommand{\Mb}{\overline{\M}}

\numberwithin{equation}{section}

\begin{document}

\title{An introduction to the intersection theory of the moduli space of curves}

\author{Hannah Larson}
\address{Department of Mathematics, University of California, Berkeley, Evans Hall 827, Berkeley, CA 94720}
\email{hlarson@berkeley.edu}
\thanks{
The author would like to thank an anonymous referee for several helpful comments that improved this article.
This article was written during time that the author served as a Clay Research Fellow.}

\date{\today}


\begin{abstract}
We introduce the intersection theory of the moduli space of curves and its tautological ring. We survey open questions about the tautological ring and sketch techniques for proving the Chow ring is or is not generated by tautological classes.
\end{abstract}

\maketitle

\section{Overview}

The moduli space of curves $\M_g$ is a space in which each point corresponds to a genus $g$ curve. Even better, curves with certain properties, like being hyperelliptic, correspond to subvarieties of the moduli space. In order to understand how different properties of curves interact, we'd like to understand how these subvarieties intersect each other. This is the basic motivation for studying the intersection theory of $\M_g$.

\medskip
\begin{center}
\includegraphics[width=4in]{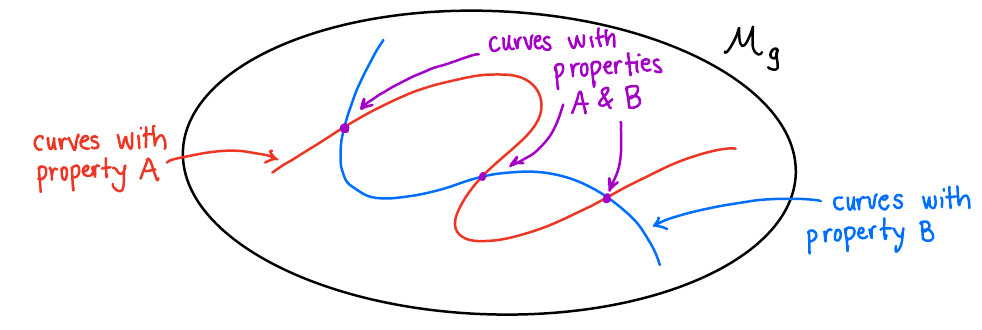}
\end{center}

In this survey article, we'll first define the \emph{Chow ring}, an invariant of a variety that keeps track of information about all its subvarieties and how they intersect each other. Then we'll see, through the example of the Grassmannian, how the Chow rings of \emph{moduli spaces} tend to inherit special classes by virtue of the geometry of their universal family. Such classes are called \emph{tautological classes}. Next, we'll define the tautological classes on $\M_g$ and survey several results and open questions about them. We are motivated particularly by the following two basic questions: When is the Chow ring generated by tautological classes? What relations exist among tautological classes? Finally, we'll take an inside look at two special cases which give a flavor of the state-of-the art tools used to address these questions. We'll sketch proofs that the Chow ring of $\M_3$ is generated by tautological classes, while the Chow ring of $\M_{12}$ is not generated by tautological classes.

\section{The Chow ring}

A basic philosophy in algebraic geometry is that we understand a variety or scheme better by understanding its subvarieties. The Chow group or Chow ring is an algebraic structure that will keep track of this information. 
A naive group that does this is the group of \emph{cycles}:
\[Z(X) := \zz \{\text{irreducible subvarities of $X$}\}. \]
This is way too big to be useful; instead we want to study subvarieties up to a meaningful equivalence relation. Two subvarities are \emph{rationally equivalent} if there is a family of subvarieties parametrized by $\pp^1$ interpolating between them:

\begin{center}
\includegraphics[width=2.5in]{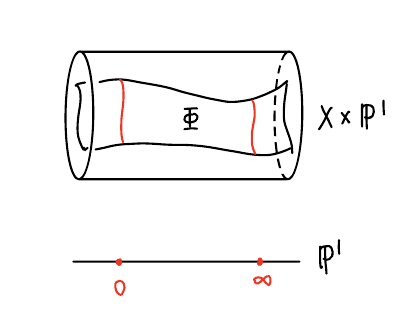}
\end{center}

The subgroup of rational equivalences is
\[\mathrm{Rat}(X) := \langle \{\Phi_\infty - \Phi_0 : \Phi \subset X \times \pp^1 \text{ with $\Phi \to \pp^1$ flat}\} \rangle \subset Z(X). \]

\begin{definition} Given a scheme $X$, its \emph{Chow group} is the group
\[\ch_*(X) = Z(X)/\mathrm{Rat}(X). \]
The lower star is the grading by dimension.
\end{definition}

If $Y$ is an irreducible subvariety, we write $[Y]$ for the equivalence class of $Y$ in $\ch_*(X)$; it is called the \emph{fundamental class of $Y$}. We extend this to reducible subvarities by declaring 
\[[Y_1 \cup Y_2] = [Y_1] + [Y_2].\]

The Chow groups have additional structure when $X$ is smooth and irreducible. In this case, we write $\ch^i(X) = \ch_{\dim X-i}(X)$ for the Chow group of codimension $i$ cycles. One can define an intersection product
\[\ch^i(X) \times \ch^j(X) \to \ch^{i+j}(X) \]
with the property that
\begin{equation} \label{gt} [A][B] = [A \cap B] \qquad \text{when $A, B$ \emph{generically transverse}}. \end{equation}
Two subvarieties are \emph{transverse} at $p$ if they are smooth along $p$ and $T_p A + T_p B = T_p X$, equivalently $\codim T_p A \cap T_p B = \codim A + \codim B$. 
\emph{Generically transverse} means this condition holds for a general $p$ in each component of $A \cap B$. We say $A$ and $B$ are \emph{dimensionally transverse} if $\codim A \cap B = \codim A + \codim B$. 

\begin{center}
\includegraphics[width=6in]{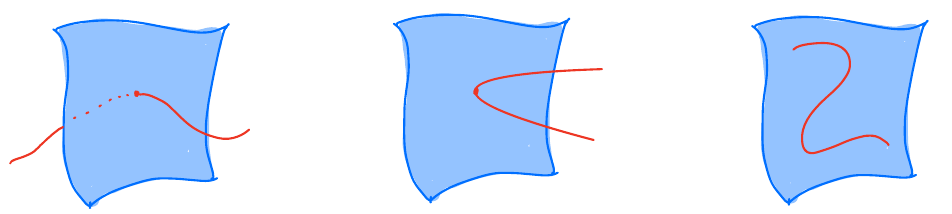}

{\color{white}aaa} \qquad transverse \qquad \qquad dimensionally transverse \qquad \qquad  \quad not dimensionally  \\
{\color{white}.} \qquad {\color{white} transverse} \qquad \qquad but not transverse \qquad \qquad \qquad \quad  transverse
\end{center}

\medskip
It's not at all obvious that an intersection product satisfying \eqref{gt} exists.
Intuitively, one might hope to prove such a product exists by establishing a ``moving lemma" which says it is always possible to express our classes in terms of linear combinations of cycles which are generically transverse. However, the solid and correct foundations for the theory rely on Segre classes and deformation to the normal cone, established by Fulton in the 1980s \cite{Fulton}.
We also refer the reader to \cite{3264} for further background on the Chow ring and enumerative problems. In the rest of this section, we briefly recall the basic properties we shall need.

\subsection{Proper pushforward}
When $f:X \subset Y$ is a closed subscheme, it's easy to see that there is a \emph{dimension-preserving} push-forward map
\[f_*: \ch_*(X) \to \ch_*(Y).\]
(View a cycle in $X$ as a cycle in $Y$. Any rational equivalence in $X$ defines a rational equivalence in $Y$.)
Less obvious is that for any \emph{proper} morphism $f: X \to Y$, there is a well-defined pushforward map
\[f_*: \ch_*(X) \to \ch_*(Y).\]
Written in terms of codimensions, this has the form
\begin{equation} \label{copush} 
f_*: \ch^i(X) \to \ch^{i-\dim X + \dim Y}(Y).
\end{equation}
Pushforward has the property that
\begin{equation} \label{pd}
f_*([A]) = \begin{cases} 0 & \text{if $\dim f(A) < \dim A$} \\ (\deg f|_A) \cdot [f(A)] & \text{if $\dim f(A) = \dim A$.}\end{cases}
\end{equation}
One extends to arbitrary cycles linearly.  
One important fact that we shall use is that if $f$ is proper and surjective, then $\ch_*(X) \otimes \qq \to \ch_*(Y) \otimes \qq$ is surjective.

\medskip
\noindent
\textit{Example:}
If $Y$ is a point, then the pushforward map $f_*: \ch_*(X) \to \ch_*(Y) = \zz$ is the \emph{degree map}. It sends a cycle of dimension zero to its degree (and vanishes on cycles of positive dimension).

\subsection{Pullback and affine bundles} Suppose now that $X$ and $Y$ are smooth.
If $f: X \to Y$ is flat, then there is a \emph{codimension-preserving} pullback map
\[f^*: \ch^*(Y) \to \ch^*(X) \]
defined by
\[f^*([B]) =[f^{-1}(B)].\]
Again, we extend to arbitrary cycles linearly. It's not hard to see that this respects rational equivalence when $f$ is flat: for any $\Phi \subset Y \times \pp^1$ giving a rational equivalence on $Y$, we have
 $(f \times 1)^{-1}(\Phi) \subset X \times \pp^1$ induces a rational equivalence of the preimages. The fact that this is a ring homomorphism follows from the familiar fact that $f^{-1}(A \cap B) = f^{-1}(A) \cap f^{-1}(B)$.
 
 Less obvious is that there is a well-defined pullback map along an arbitrary morphism of smooth varieties $f: X \to Y$ (or even when $X \to Y$ is a proper lci morphism). It has the property that 
 \[f^*([B]) = [f^{-1}(B)] \]
whenever $B$ is ``generically transverse to $f$." Showing this pullback is well-defined requires the moving lemma and some work. 

An important case of note is when $X \to Y$ is the total space of a vector bundle. In this case, the pullback map $\ch^*(Y) \to \ch^*(X)$ is an isomorphism. A special case of this is that $\ch^*(\mathbb{A}^n) = \ch^*(\mathrm{pt}) = \zz$. Here, we mean $\zz$ in degree $0$, generated by the fundamental class, and all other Chow groups vanish.

\subsection{The push-pull formula and self-intersection} \label{si}
If $f: X \to Y$ is a proper morphism between smooth schemes, then we have pushforward and pullback and they interact well:
\[ f_*([A]) \cdot [B] = f_*([A] \cdot f^*[B]).\]
This should feel intuitive on the level of subvarieties: the [image of $A$] intersected with $B$ should be the same as the image of [$A$ intersected with the preimage of $B$].

\medskip
\noindent
\textit{Example:} Suppose $\iota\colon D \subset X$ is a divisor. We have $[D] = c_1(\O(D)) \in \ch^1(X)$, so $[D]^2 = c_1(\O(D))^2$. Something better is true though: the answer is actually the pushforward of a class supported on $D$. Indeed, $[D] = \iota_*[1]$, so by the push-pull formula,
\begin{equation} [D]^2 = (\iota_*[1]) \cdot D = \iota_*(\iota^*[D])  \label{self-int}
\end{equation}
Here, $\iota^*[D] = c_1(\O(D)|_D)$. You may recognize the latter as $c_1(N_{D/X})$. This is an example of the \emph{self-intersection formula}, which in turn is a special case of the \emph{excess interseciton formula}, see e.g. \cite[Chapter 13]{3264}.

\subsection{Excision} \label{sec:ex}
Suppose $Z \subset X$ is a closed subscheme and $U = X \smallsetminus Z$ is the open complement. Then there is a right-exact sequence
\[\ch_{*}(Z) \rightarrow \ch_*(X) \rightarrow \ch_*(U) \rightarrow 0. \]
The first map is proper push forward and the second map is a flat pullback.
When $Z \subset X$ is smooth of codimension $c$, we'll also write
\[\ch^{*-c}(Z) \rightarrow \ch^*(X) \rightarrow \ch^*(U) \rightarrow 0. \]
Take care that the second map is a ring homomorphism but the first is not.

Many moduli spaces of interest, like the moduli space of curves and its compactification, admit geometrically meaningful stratifications. Excision implies that the Chow group is generated by lifts of the pushforwards of classes from each piece of the stratification. However, understanding the relations among these generators and the ring structure may be very challenging, and is sometimes
referred to as a ``patching problem."  One particularly nice case which is well understood is when $X$ admits an affine stratification.

\begin{lem} \label{affstrat}
If $X$ admits an affine stratification $X = \bigsqcup_i U_i$ where each $U_i \cong \mathbb{A}^{n_i}$, then $\ch_*(X)$ is freely generated as a group by $[\overline{U}_i]$.
\end{lem}
Excision and the fact that $\ch^*(\mathbb{A}^n) = \zz$ show that the fundamental classes $[\overline{U}_i]$ generate $\ch_*(X)$. In general, the excision sequence is not left-exact, so it is not as clear that the classes $[\overline{U}_i]$ are all independent. However, this can be established using the cycle class map, which is an isomorphism for such spaces, see \cite[Examples 1.9.1 and 19.1.11]{Fulton}.

\medskip
The other way excision is used when we know $\ch^*(X)$ and we want to find $\ch^*(U)$.
This is often easier than a patching problem, but solving ``excision problems", i.e. finding the image of $\ch_*(Z) \to \ch_*(X)$, can still be hard.

\section{Grassmannians and $\BGL_k$}
Projective space, and more generally Grassmannians, admit affine stratifications, so Lemma \ref{affstrat} determines generators for their Chow groups. The more interesting part becomes identifying the \emph{ring} structure. After discussing these classical examples, we'll also discuss how to use them to make sense of the Chow ring of the classifying stack $\BGL_k$.

\subsection{Projective space}
Projective space has an affine stratificaiton
\begin{equation} \label{as}
\pp^n = \aa^n \sqcup \aa^{n-1} \sqcup \cdots \sqcup \aa^1 \sqcup \mathrm{pt}. 
\end{equation}
By excision the group $\ch_*(\pp^n)$ is generated by the fundamental classes of their closures:
\begin{equation} \label{gens}
 \zz[\pp^n] \oplus \zz[\pp^{n-1}] \oplus \cdots \zz[\pp^1] \oplus \zz[\mathrm{pt}] \twoheadrightarrow \ch_*(\pp^n).
 \end{equation}
Let $\zeta = [\pp^{n-1}]$ be the class of a hyperplane. Each $[\pp^{n-k}]$ is the transverse intersection of $k$ hyperplanes, so $[\pp^{n-k}] = \zeta^k$, down to $\zeta^n$ is the class of a point, and $\zeta^{n+1} = 0$. Since $\zeta^n \neq 0$, each of the $\zeta^k \neq 0$, so \eqref{gens} is an isomorphism.
As a \emph{ring}, we have
\[\ch^*(\pp^n) = \zz[\zeta]/(\zeta^{n+1}). \]

\subsection{The Grassmannian} \label{gkn}
The Grassmannian $G(k,n)$ is the moduli space of $k$-dimensional subspaces of an $n$-dimensional vector space. 

Generalizing \eqref{as},
the Grassmannian $G(k,n)$ admits an affine stratification into \emph{Schubert cells}. The cells of codimension $i$ are labeled by partitions $\lambda = (\lambda_1, \ldots, \lambda_k)$ with $\lambda_1 + \ldots + \lambda_k = i$ and $\lambda_j \leq n - k$. This is the archetypal example of a stratification of a moduli space.

Identifying $G(k,n)$ with the space of full rank $k \times (n - k)$ matrices modulo $\GL_k$,
each cell $\Sigma_{\lambda}^{\circ}$ can be represented by matrices in reduced row echelon form with a particular collection of pivots. For example in $G(2,4)$, the large open cell is 
\[\Sigma_{(0,0)}^\circ= \left\{ \mathrm{span} \left(\begin{matrix} 1 & 0 & * & *  \\ 0 & 1 & * & *\end{matrix} \right) \right\}. \]
The free entries $*$ make up the affine space. 
The cell for partition the $(2, 1)$ is obtained by moving the pivots over by $2$ and $1$ starting from the bottom:
\[\Sigma_{(2,1)}^\circ= \left\{ \mathrm{span} \left(\begin{matrix} 0 & 1 & * & 0  \\ 0 & 0 & 0 & 1\end{matrix} \right) \right\}. \]
 The closure $\Sigma_{(2,1)}$ is the subvariety of $2$-dimensional subspaces that are contained in $V(x_1)$ and meet $V(x_1, x_2, x_3)$ non-trivially, which is the union of $\Sigma_{2,1}^\circ$ and $\Sigma_{2,2}^\circ$.
In general, the closure of $\Sigma_\lambda^\circ$ is the union of Schubert cells $\Sigma_\mu^\circ$ where the Young diagram of $\mu$ contains that of $\lambda$.

By Lemma \ref{affstrat} the classes of the closures of the cells, $\Sigma_{\lambda}$, freely generate $\ch_*(G(k,n))$ as a group. 
The ring structure 
is determined by structure constants $c_{\lambda \mu}^{\nu}$ such that
\[[\Sigma_{\lambda}] [\Sigma_{\nu}] = \sum c_{\lambda \mu}^{\nu} [\Sigma_{\nu}].\] 
The $c_{\lambda \mu}^{\nu}$ are called \emph{Littlewood--Richardson coefficients}, and multiplying these classes is called \emph{Schubert calculus}.  See \cite[Chapter 4]{3264} for a nice exposition.

\subsection{Tautological classes}
The classes of Schubert varieties give us additive generators for $\ch^*(G(k,n))$, but far fewer classes are needed to generate $\ch^*(G(k,n))$ as a ring. It turns out that one nice collection of generators are the fundamental classes $[\Sigma_\lambda]$ where $\lambda$ has one part. These are called \emph{special Schubert varieties}. If $\lambda$ has one part, then only the last pivot is sent over to the right.
This tells us that $\Sigma_p$ is the locus of $k$-dimensional subspaces $\Lambda$ that meet a fixed $n - k -p + 1$-dimensional subspace non-trivially:
\[\Sigma_p = \{\Lambda : \Lambda \cap F_{n-k-p+1} \neq 0\}. \]
The \emph{Giambelli formula} expresses each $[\Sigma_{\lambda}]$ in terms of special Schubert classes, so special Schubert classes are ring generators for $\ch^*(G(k,n))$.

How do we understand the relations among these multiplicative generators? The answer comes from relating the special Schubert classes to the universal family:
\[0 \rightarrow \mathcal{S} \rightarrow \O^{\oplus n} \rightarrow \Q \rightarrow 0. \]
Suppose that $F_{n - k - p + 1} = \langle v_1, \ldots, v_{n - k - p + 1} \rangle$. Each vector $v_i$ corresponds to a coordinate section of $\O^{\oplus n}$, which gives rise to a section $\overline{v}_i$ of $\Q$. Typically the sections $\overline{v}_i$ of $\Q$ are independent in $\Q$, but they become dependent precisely when $\Lambda \cap \langle v_1, \ldots, v_{n - k - p + 1} \rangle \neq 0$.
\[\Sigma_p = \text{locus where $\overline{v}_1, \ldots, \overline{v}_{n - k - p + 1}$ become dependent}. \]
The locus where $n -k - p + 1 = \rank(\Q) - p + 1$ sections of $\Q$ become dependent is the Chern class $c_p(\Q)$. (If you haven't seen Chern classes before, you can take this as the definition; more generally, this works to define the Chern classes of any globally generated vector bundle. See \cite[Chapter 5]{3264} for an introduction to Chern classes.)

The Chern classes of $\Q$ are our first example of a \emph{tautological class}.  This term was coined by Mumford in his 1983 paper \cite{mum}, where he wrote

\medskip
\emph{``Whenever a variety or topological space is defined by some universal property, one expects that by virtue of its defining property, it possesses certain classes called tautological classes."}
\smallskip

In our case, the Grassmannian is defined by a universal property (a map $X \to G(k, n)$ is the same as the data of a rank $n-k$ quotient of a rank $n$ trivial bundle on $X$) and $\Q$ is the universal family of such quotients (each rank $n-k$ quotient of a rank $n$ trivial bundle on $X$ is the pullback of $\Q$ under the corresponding map $X \to G(k, n)$).
The Chern classes of $\Q$ therefore fit Mumford's idea of a tautological class. In general, to make precise statements, one has to define what the tautological classes on a moduli space are. There are generally accepted definitions of tautological classes for well-studied spaces like Grassmannians, moduli spaces of curves \cite{mum}, abelian varieties \cite{vd}, and K3 surfaces \cite{MOP}. They typically involve finding a vector bundle and taking Chern classes, but may also involve pushing forward and pulling back along natural maps.

Chern classes play well in short exact sequences, and one consequence of this is that
\[\frac{1}{1 + c_1(\Q) + \ldots + c_{n - k}(\Q)} = 1 + c_1(\mathcal{S}) + \ldots + c_k(\mathcal{S}). \]
The expression on the left is formally inverted using the identity $1/(1 + x) = 1 - x + x^2 - \ldots$.
The right-hand side tells us that, when we perform this expansion, all  terms in degree above $k$ vanish. It turns out that these generate all relations among $c_1(\Q), \ldots, c_{n - k}(\Q) \in \ch^*(G(k,n))$.
Thus, we conclude with a very nice description of the Chow ring of $G(k,n)$. It is generated tautological classes, and we can understand the relations geometrically (they all come from the fact that the kernel of the quotient has rank $k$ and therefore its Chern classes vanish in degree $> k$):
\[\ch^*(G(k, n)) = \zz[c_1(\Q), \ldots, c_{n-k}(\Q)]/
\left\langle \left \{\frac{1}{1 + c_1(\Q) + \ldots + c_{n - k}(\Q)} \right\}^j : j > k \right\rangle.
\]
We could also write the Chow ring in terms of the Chern classes of the tautological subbundle:
\begin{equation} \label{sub} \ch^*(G(k, n)) = \zz[c_1(\mathcal{S}), \ldots, c_{k}(\mathcal{S})]/
\left\langle \left \{\frac{1}{1 + c_1(\mathcal{S}) + \ldots + c_{k}(\mathcal{S})} \right\}^j : j > n - k \right\rangle.
\end{equation}
These two presentations are related to each other via the isomorphism between $G(k, n)$ and $G(n-k,n)$ obtained by taking duals.

\subsection{The moduli stack of vector bundles}
The theory of Chow rings has been extended to smooth algebraic stacks admitting a stratification by quotient stacks \cite{Kresch}. Here, we explain one of the key ideas developed by Edidin and Graham \cite{EG}, which applies to quotients of algebraic varieties (or algebraic spaces) by linear algebraic groups. In essence, excision and the isomorphism for Chow rings of affine bundles allows us to figure out what the Chow rings of such quotient stacks should be.

Suppose $\mathcal{X}$ is a stack and $\mathcal{V} \to \mathcal{X}$ is a vector bundle. We would certainly like it to be true that $\ch^*(\mathcal{X}) \cong \ch^*(\mathcal{V})$. Suppose $\mathcal{U} \subset \mathcal{V}$ is an open substack whose complement has codimension $c$. By the excision sequence, we would certainly like it to be true that $\ch^i(\mathcal{V}) \cong \ch^i(\mathcal{U})$ for $i < c$.
Now if $\mathcal{U}$ happens to be a scheme, we already know what we mean by its Chow groups and we could declare that $\ch^i(\mathcal{X}) = \ch^i(\mathcal{U})$. We think of $\mathcal{U}$ as a ``model" of $\mathcal{X}$: Up to affine bundles and loci of high codimension --- which we know only changes intersection theory in high degrees --- $\mathcal{U}$ behaves the same as $\mathcal{X}$. 

We can make this into a definition if we show it is independent of choices. Indeed, if we picked another vector bundle $\mathcal{V}'  \to \mathcal{X}$ and an open substack $\mathcal{U}' \subset \mathcal{V}'$ whose complement has codimension $c' > i$, then we find the following equalities of Chow groups of schemes:
 \[ \ch^i(\mathcal{U}') = \ch^i(\mathcal{U}' \times_{\mathcal{X}} \mathcal{V}) =  \ch^i(\mathcal{U}' \times_{\mathcal{X}} \mathcal{U}) = \ch^i(\mathcal{V}' \times_{\mathcal{X}} \mathcal{U})
 = \ch^i(\mathcal{U}).
  \]
Equalities hold at each step because either one space is a vector bundle over the other, or the complement of one in the other has codimension greater than $i$.
We can also use this to determine the multiplicative structure. Since any two classes live in fixed finite degree we just need to find a model $\mathcal{U}$ that approximates $\mathcal{X}$ beyond the sum of their degrees.

Let's see how this works in action with $\BGL_k = [\mathrm{pt}/\GL_k]$, the moduli stack of rank $k$ vector bundles. Consider the vector space $\mathbb{A}^{k(n-k)}$ of $k \times (n - k)$ matrices. Let $\GL_k$ act by left multiplication. The quotient stack $[\mathbb{A}^{k(n-k)}/\GL_k]$ is a vector bundle over $\BGL_k$. Now consider the open subset $U \subset \mathbb{A}^{k(n-k)}$ of full rank matrices. The complement of $U$ is the determinantal locus of matrices of rank $< k$ and so has codimension $(\dim \ker)(\dim \coker) = n - k + 1$. Taking the quotient by $\GL_k$, we obtain an open substack $[U/\GL_k] \subset [\mathbb{A}^{k(n-k)}/\GL_k]$ whose complement has codimension $n - k + 1$. But $[U/\GL_k]$ is the Grassmannian, so we know its Chow groups! Hence, for any $i < n - k + 1$, we have $\ch^i(\BGL_k) = \ch^i(G(k, n))$. Now look back at \eqref{sub}. In degrees $i < n - k + 1$, $\ch^i(G(k,n))$ agrees with the free polynomial ring $\zz[c_1, \ldots, c_k]$. This holds for any value of $n$. Taking larger and larger $n$, we conclude that 
\[\ch^*(\BGL_k) = \zz[c_1, \ldots, c_k].\]
The classes $c_1, \ldots, c_k$ are the universal Chern classes. Again, Mumford's philosophy shines through. The stack $\BGL_k$ has a universal rank $k$ vector bundle on it and the Chern classes of this universal bundle give rise to distinguished classes in the Chow ring. As was the case for the Grassmannian, these tautological classes generate the Chow ring. Unlike for the Grassmannian, there are no relations. 

This discussion shows that we can think of the Grassmannian as a sort of ``finite incarnation" or approximation of $\BGL_k$.
Note that the relations among $c_1(\mathcal{S}), \ldots, c_k(\mathcal{S})$ in the Grassmannian $G(k, n)$ start in codimension $n - k + 1$, exactly the codimension of the complement of the full rank matrices inside the space of all matrices.

\section{Moduli spaces of curves}
\subsection{Comment on rational coefficients and stacks} \label{comment}
Starting now, we will work with Chow rings with rational coefficients, i.e. $\ch^i(X)$ now means $\ch^i(X) \otimes \qq$. The integral intersection theory of $\M_g$ is also an active area of research, but it can be quite subtle. For example, the integral Chow ring of $\M_g$ is only known for $g =2$ \cite{Vistoli}, whereas with rational coefficients it is known for $g \leq 9$ \cite{mum,F2,F3,Iz,PV,789}. When we work with rational coefficients we are allowed to conflate a Deligne--Mumford stack with its coarse moduli space since they have the same rational Chow ring.
This is not true integrally though. For example, for all $i > 0$, the integral Chow group $\ch^i(\mathrm{B}(\zz/2)) = \zz/2$. Meanwhile, the coarse moduli space of  $\mathrm{B}(\zz/2)$ is just a point, which has $\ch^i(\mathrm{pt}) = 0$ for all $i > 0$. Nevertheless, once we tensor with $\qq$ these two agree, as promised.

\subsection{Tautological classes}
Consider the universal curve $f: \C \to \M_g$. 
\begin{center}
\includegraphics[width=2in]{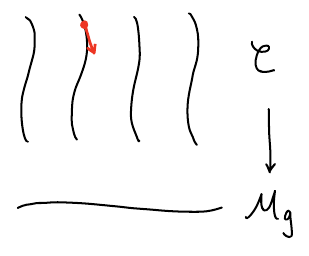}
\end{center}
Unlike the Grassmannian we don't have a vector bundle staring us in the face, but when we see something nonlinear, we should try to find a linear approximation. Consider the relative dualizing sheaf $\omega_f$, which, as $f$ is smooth, agrees with the relative cotangent bundle. This line bundle is dual to the line bundle on $\C$ whose fiber at each point is the vertical tangent space through $p$, pictured in red above.

Taking Chern classes, we obtain $c_1(\omega_f) \in \ch^1(\C)$.  In order to get a classes on $\M_g$, we use proper pushforward. If we just push forward $c_1(\omega_f)$ we get a class in degree $0$, see \eqref{copush}. In fact, using \eqref{pd} the class we get is $2g - 2$ (the degree of $\omega_f$ on fibers) times the fundamental class of $\M_g$. To get classes in higher codimension, we use the multiplicative structure in $\ch^*(\C)$ to take powers before pushing forward. We define the \emph{kappa classes}
\[\kappa_j := f_*(c_1(\omega_f)^{j+1}) \in \ch^i(\M_g). \]

There is another natural way to cook up a vector bundle from the universal curve, which is to pushforward $\omega_f$. By cohomology and base change, $\mathbb{E} := f_*\omega_f$ is a vector bundle of rank $g$ whose fiber the point $[C] \in \M_g$ is $H^0(\omega_C)$. This vector bundle is called the \emph{Hodge bundle} and its Chern classes are called the \emph{lambda classes}
\[\lambda_i := c_i(\mathbb{E}).\]

It turns out that the lambda classes are expressible in terms of the kappa classes.
This is a consequence of \emph{Grothendeick--Riemann--Roch}, which relates the Chern classes of a pushforward of a sheaf to pushforwards of expressions involving Chern classes of the original sheaf, see \cite[Chapter 14]{3264}. 
This formula involves the Chern character, $\mathrm{Ch}$, and Todd class, $\mathrm{Td}$, which are given by polynomial expressions in terms of usual Chern classes.
In our context, the formula says
\begin{equation} \label{grr} \mathrm{Ch}(\mathbb{E}) - \mathrm{Ch}(\O) = \mathrm{Ch}(f_*\omega_f) - \mathrm{Ch}(R^1f_*\omega_f) = f_*(\mathrm{Ch}(\omega_f) \mathrm{Td}(\omega_f^\vee)). 
\end{equation}
Expanding the right-hand side we obtain:
\[f_*\left(\left( 1 + c_1(\omega_f) + \frac{1}{2}c_1(\omega_f)^2 + \ldots \right) \left( 1 - \frac{1}{2}c_1(\omega_f) + \frac{1}{12} c_1(\omega_f)^2 + \ldots \right) \right), \]
from which we see it involves pushforwards of powers of $c_1(\omega_f)$, i.e. kappa classes. The term on the right-hand side of \eqref{grr} in degree $1$ comes from pushing forward the degree $2$ term of the expression in parentheses. Multiplying this out and collecting terms, we find
\[\lambda_1 = c_1(\mathbb{E}) = \mathrm{Ch}_1(\mathbb{E}) = f_*\left( \frac{1}{12} c_1(\omega_f)^2 \right) = \frac{1}{12} \kappa_1.\]
The rest of the lambda classes can similarly be expressed explicitly in terms of kappa classes, see \cite[Corollary 6.2]{mum}.

Much like we've seen with the lambda classes, lots of other ways one might think to try to construct classes in $\ch^*(\M_g)$ from the universal family (e.g. take tensor powers of $\omega_f$ and then push forward) all yield classes expressible in terms of the kappa classes. This is part of the motivation behind the following definition.

\begin{definition} \label{taut-def}
The tautological ring $\mathsf{R}^*(\M_g) \subset \ch^*(\M_g)$ is the subring generated by the kappa classes $\kappa_1, \kappa_2, \ldots$.
\end{definition}
\noindent
Thinking back to the Grassmannian, we now ask two questions: 
\begin{enumerate}
\item What are the relations among the tautological classes? Particularly, can we find a geometric explanation for them? 
\item Do the tautological classes generate the entire Chow ring? If not, what can we say about non-tautological classes?
\end{enumerate}

\subsection{Tautological relations} \label{tr}
The study of tautological relations has a long and interesting history. In \cite{F1}, Faber gave a geometric machine for producing relations among the kappa classes, which we will explain below. The idea is reminiscent of ideas we saw for the Grassmannian.

Consider the $d$th power of the universal curve $\pi: \C^d \to \M_g$. Each point on $\C^d$ corresponds to a curve $C$ together with $d$ (not necessarily distinct) points $p_1, \ldots, p_d$ on it. Let $\Omega_d \to \C^d$ be the rank $d$ vector bundle whose fiber at a point $[C, p_1, \ldots, p_d]$ is $H^0(C, \omega_C|_{p_1+\ldots + p_d})$. We will also make use of $\pi^*\mathbb{E}$, whose fiber at $[C, p_1, \ldots, p_d]$ is $H^0(\omega_C)$. Now consider the map
\begin{equation} H^0(\omega_C) \to H^0(\omega_C|_{p_1 + \ldots + p_d}). \label{fh0}
\end{equation}
The kernel is $H^0(\omega_C(-p_1 - \ldots - p_d))$ which vanishes if $d > 2g - 2$. If $d > 2g - 2$, then \eqref{fh0} globalizes to an injective map of vector bundles $\pi^*\mathbb{E} \rightarrow \Omega_d$ on $\C^d$. Thus, we obtain a short exact sequence
\[0 \rightarrow \pi^*\mathbb{E} \rightarrow \Omega_d \rightarrow Q \rightarrow 0\]
where the quotient $Q$ is locally free with $\rank Q = \rank \Omega_d - \rank \pi^*\mathbb{E} = d - g$.
In particular its Chern classes vanish in degree above its rank:
\begin{equation} \label{ciQ} 0 = c_i(Q) = \left\{\frac{1 + c_1(\Omega_d) + \ldots + c_d(\Omega_d)}{1 + \pi^*\lambda_1 + \ldots + \pi^*\lambda_g} \right\}^i \qquad \qquad \text{for $i > d - g$.}
\end{equation}
This gives relations on $\C^d$, which we can use to produce relations on $\M_g$. Indeed,
we can multiply the $c_i(Q)$ by diagonals in $\C^d$ and pullbacks of $c_1(\omega_f)$ along projections $\C^d \to \C$, and then push the result forward to $\M_g$. As you might imagine, this results in many non-trivial polynomials in the kappa classes. These will all be relations in the tautological ring whenever $i > d - g$. To see why this actually gives tautological classes, we need dig into the Chern classes of $\Omega_d$. The reader wishing to skip technical details should skip ahead to Section \ref{con}.

\subsubsection{Chern classes of $\Omega_d$}
Above we described $\Omega_d$ by its fibers, but to rigorously construct this vector bundle and understand its Chern classes, we will use cohomology and base change. Let $\alpha: \C^{d+1} \to \C^d$ be projection onto the first $d$ factors. We will write $D_{i,j} \subset  \C^{d+1}$ for the diagonal where the $i$th and $j$th points agree. The diagonals $D_{i,d+1}$ determine sections of $\alpha$, and $\alpha$ equipped with these sections is the universal curve with $d$ (not necessarily distinct) marked points over $\C^d$.

Now consider the following sequence of sheaves on $\C^{d+1}$:
\begin{equation} \label{Fseq} 0 \rightarrow \omega_\alpha(-D_{1,d+1} - \cdots - D_{d,d+1}) \rightarrow \omega_\alpha \rightarrow \F := \omega_\alpha|_{D_{1,d+1} + \cdots + D_{d,d+1}} \rightarrow 0.
\end{equation}
When we apply pushforward by $\alpha$, we obtain the long exact sequence of sheaves on $\C^d$:
\[0 \rightarrow  \alpha_*\omega_\alpha(-D_{1,d+1} - \cdots - D_{d,d+1}) \rightarrow \alpha_*\omega_\alpha \rightarrow \alpha_*\F \rightarrow \cdots \]
By cohomology and base change, the fibers of this sequence over $[C,p_1, \ldots, p_d]$ are
\[0 \rightarrow H^0(\omega_C(-p_1 - \cdots -p_d)) \rightarrow H^0(\omega_C) \rightarrow H^0(\omega_C|_{p_1 + \ldots + p_d}) \rightarrow \cdots \]
In particular, $\alpha_*\omega_\alpha(-D_{1,d+1} - \cdots - D_{d,d+1}) = 0$, we recognize $\alpha_*\omega_\alpha = \pi^*\mathbb{E}$, and finally, $\alpha_*\F$ is $\Omega_d$.
This description of $\Omega_d$ is helpful because 
we can use Grothendieck--Riemann--Roch to determine its Chern classes. Note that $R^1\alpha_* \F = 0$, since $\F$ is supported on the sections and hence, $H^1$ of its restrictions to fibers of $f$ vanish.
Thus, we have
\begin{equation} \label{rr} \mathrm{Ch}(\Omega_d) = \mathrm{Ch}(\alpha_*\F) - \mathrm{Ch}(R^1\alpha_*\F) =\alpha_*(\mathrm{Ch}(\F) \mathrm{Td}(\omega_\alpha^\vee)). 
\end{equation}
Let us write $K_i =\mathrm{pr}_i^* c_1(\omega_f)$. For example, $c_1(\omega_{\alpha}) = K_{d+1}$. Now, the sequence \eqref{Fseq} determines the Chern classes of $\F$:
\[c(\F) =\frac{c(\omega_\alpha)}{c(\omega_\alpha(-D_{1,d+1} - \cdots - D_{d,d+1})} = \frac{1 + K_{d+1}}{1 + K_{d+1} - D_{1,d+1} - \cdots - D_{d,d+1}}.\]
In the last term above, to avoid cluttered notation, we are also writing $D_{i,j}$ for its fundamental class.
In particular, the above equation implies that the term in each degree of the right hand side of \eqref{rr} is a polynomial in the $K_i$ and $D_{ij}$.

\begin{lem}[see p.\ 10 of \cite{F1}] \label{apush}
The $\alpha$ pushforward of any polynomial in the $K_i$ and $D_{ij}$  (with $1 \leq i, j \leq d+1$)  is a polynomial in the $K_i$ and $D_{ij}$ (with $1 \leq i, j \leq d$) and $\pi^*\kappa_i$.
\end{lem}
\begin{proof}
Since the $D_{i,j}$ with $i, j < d+1$ are pulled back along $\alpha$, using the push-pull formula, it suffices to study the $\alpha$ pushforwards of polynomials in $K_{d+1}$ and $D_{i,d+1}$.
Consider a monomial $K_{d+1}^a D_{1,d+1}^{b_1} \cdots D_{d,d+1}^{b_d}$.  If all $b_i = 0$, then 
\[\alpha_* K_{d+1}^a = \pi^* \kappa_{a-1}\]
and we are done. Now suppose some $b_i \neq 0$. Let $\iota: D_{i,d+1} \to \C^{d+1}$ be the inclusion. By the self-intersection formula (see Section \ref{si}), 
\[\iota^*D_{i,d+1} = \iota^*\iota_*1 = c_1(N_{D_{i,d+1}/\C^{d+1}}) = -\iota^*K_{d+1}.\]
Above, the last equality uses that $\iota$ is a section of $\alpha$, so the normal bundle to it is the restriction of the relative tangent bundle of $\alpha$, see the picture below.
\begin{equation} \label{normal}
\includegraphics[width=2in]{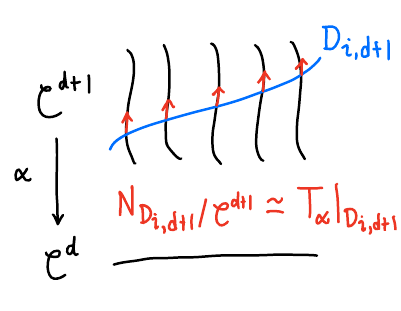}
\end{equation}
Repeated application of the push-pull formula now yields
\begin{align*}
D_{i,d+1}^{b_i} &= (\iota_*1)D_{i,d+1}^{b_i-1} = \iota_*(\iota^*D_{i,d+1}^{b_i-1}) = \iota_*(\iota^*(-K_{d+1})^{b_i - 1}) \\
&= (\iota_*1)(-K_{d+1})^{b_i-1} = D_{i,d+1}(-K_{d+1})^{b_i-1}.
\end{align*}

Making this substitution, it suffices to treat the case that each $b_i$ equals $1$ or $0$.
After relabeling, we may assume that the indices $i$ for which $b_i =1$ are the first $j$ indices. Thus, our task is to compute
$\alpha_*(K_{d+1}^a D_{1, d+1} \cdots D_{j, d+1})$.
Let $\iota: \C^d \cong D_{1,d+1} \to \C^{d+1}$ be the inclusion. Since $D_{1,d+1}$ is the locus where the $1$st and $d+1$st marking agree, we have
\[\iota^*D_{i,d+1} = D_{1,i} \qquad \text{and} \qquad \iota^* K_{d+1} = K_1.\]

Then, we compute
\begin{align*} \alpha_*(K_{d+1}^a D_{1, d+1} \cdots D_{j, d+1}) &= \alpha_*(K_{d+1}^a (\iota_* 1)D_{2, d+1} \cdots D_{j, d+1}) \\
&= \alpha_*(\iota_* \iota^*(K_{d+1}^aD_{2, d+1} \cdots D_{j, d+1})) \\
&= \alpha_*\iota_*(K_1^a D_{1,2} \cdots D_{1,j}) \\
&=K_1^a D_{1,2} \cdots D_{1,j}.
\end{align*}
For the last equality above, recall that $\alpha_*\iota_*$ is the identity because $\iota$ is a section.
\end{proof}

Repeated application of Lemma \ref{apush} shows that the pushforward along $\pi$ of any polynomial in the $K_i$ and $D_{ij}$ is a tautological class on $\M_g$. (At the last step, once one marking is left, there is only $K_1$ and pushforwards of polynomials in $K_1$ give the kappa classes by definition.) Equation \eqref{ciQ} gives a formula for $c_k(Q)$ as a polynomial in Chern classes of $\Omega_d$ and $\pi^* \lambda_j$. Moreover, the Chern classes of $\Omega_d$ are polynomials in the $K_i, D_{ij}$ and $\pi^*\kappa_i$. 
Taking the product of $c_k(Q)$ with any polynomial in the $K_i$ and $D_{ij}$ and pushing forward by $\pi$ thus yields a tautological class on $\M_g$. Such expressions are explicitly computable using the ideas in Lemma \ref{apush} and often yield non-trivial polynomials in the kappa classes.
On the one hand, the fact that $c_k(Q) = 0$ for $k > d - g$ means such a class is zero whenever $k > d - g$. This is how Faber produced many relations among tautological classes.

\subsubsection{Faber's conjecture} \label{con}

Based on the relations produced in low genus examples, Faber made the following conjecture concerning the structure of the tautological ring.

\begin{conj}[Faber \cite{F1}]
The tautological ring $\mathsf{R}^*(\M_g)$ satisfies the following properties:
\begin{enumerate}
\item It is generated by $\kappa_1, \ldots, \kappa_{\lfloor g/3 \rfloor}$, which satisfy no relations in degrees $\leq \lfloor g/3 \rfloor$;
\item It vanishes in degrees $> g - 2$ and $\mathsf{R}^{g-2}(\M_g) = \qq$;
\item There is a perfect pairing
\[\mathsf{R}^i(\M_g) \times \mathsf{R}^{g - 2 - i}(\M_g) \to \mathsf{R}^{g-2}(\M_g) = \qq.\]
\end{enumerate}
\end{conj}

Part (1) was proved by Ionel \cite{Ionel} and part (2) was proved by Looijenga \cite{Loo}.

\medskip
Based on further studying these relations, Faber and Zagier conjectured an explicit set of relations for all $g$, now called the \emph{Faber--Zagier relations}. These relations are specified in terms of a generating function involving hypergeometric series, logarithms, and exponentials. 
Years later, Pixton and Pandharipane proved that all of the Faber--Zagier relations are indeed relations \cite{PP}. Their construction of relations is also geometric in nature, but not quite as simple as the one above.
It still remains an open problem if Faber's original relations, produced by the recipe above, generate all Faber--Zagier relations. It also remains an open problem if the Faber--Zagier relations are all relations. We refer the reader to the excellent survey article \cite{calc} for a description of the Faber--Zagier relations and their history.

To further discuss (3), let us define $\mathsf{R}_{\mathsf{FZ}}(\M_g)$ to be the polynomial ring in kappa classes modulo the Faber--Zagier relations. Then there is a surjection $\mathsf{R}_{\mathsf{FZ}}^*(\M_g) \to \mathsf{R}^*(\M_g)$, which is an isomorphism if the Faber--Zagier relations are complete.
All proportionalities among the tautological classes in degree $g - 2$ are known, conjectured in \cite{F1} and proven in \cite{FP,GP}.
Thus, in theory one can compute the pairing
\begin{equation} \label{fzp}
\mathsf{R}_{\mathsf{FZ}}^i(\M_g) \times \mathsf{R}_{\mathsf{FZ}}^{g - 2 - i}(\M_g) \to \mathsf{R}_{\mathsf{FZ}}^{g-2}(\M_g) = \qq.
\end{equation}
If the pairing is perfect, then there cannot be any more relations. Indeed, a class cannot be zero if its product with some other class landing it in degree $g - 2$ is nonzero.
In this way, Faber computationally proved that $\mathsf{R}_{\mathsf{FZ}}^*(\M_g) = \mathsf{R}^*(\M_g)$ for $g \leq 23$ and that the pairing is perfect. However, with improved computing power, the case
$g = 24$ was eventually completed, and it showed that \eqref{fzp} is \emph{not} perfect. Thus, either part (3) is false or the Faber--Zagier relations are incomplete. The general belief these days seems to be that part (3) is false and the Faber--Zagier relations are complete.

Evidence in this direction comes from studying the analogous question on the compactification $\Mb_{g,n}$, parameterizing genus $g$ stable curves with $n$ marked points. The tautological ring of $\Mb_{g,n}$ (defined in Section \ref{bart}) satisfies
 $\mathsf{R}^{3g - 3+n}(\Mb_{g,n}) = \qq$, and the tautological ring vanishes in higher degrees as the moduli space has dimension $3g - 3 + n$. Thus it is natural to study the pairing 
\begin{equation} \label{bp} 
\mathsf{R}^{i}(\Mb_{g,n}) \times \mathsf{R}^{3g - 3+n-i}(\Mb_{g,n})  \to
\mathsf{R}^{3g - 3+n}(\Mb_{g,n}) = \qq.
\end{equation}
In \cite{PT}, Petersen and Tommasi show that the pairing fails to be perfect when $g = 2$ and $n$ is sufficiently large. Further investigation by Petersen \cite{P} shows that the pairing is not perfect when $g = 2$ and $n \geq 20$. This result was recently extended to higher genera by Canning \cite{C}.

\begin{thm}[Canning \cite{C}]
If $g \geq 2$ and $2g + n \geq 24$, then the pairing \eqref{bp} is not perfect.
\end{thm}

\subsection{The tautological generation question}
In recent years, there has also been great progress on the question of when $\ch^*(\M_g) = \mathsf{R}^*(\M_g)$. Through the combined efforts of many researchers, it remains open only in a small handful of genera.

The known cases where $\ch^*(\M_g) = \mathsf{R}^*(\M_g)$ are as follows.

\begin{thm}[Mumford \cite{mum}, Faber \cite{F2,F3}, Izadi \cite{Iz}, Penev--Vakil \cite{PV}, Canning--Larson \cite{789}] \label{t}
If $2 \leq g \leq 9$, $\ch^*(\M_g) = \mathsf{R}^*(\M_g)$.
\end{thm}

\noindent
Furthermore, in the above cases, the tautological ring is also known by Faber's computer calculations, so the Chow ring is completely understood. 

The moduli space of genus $g$ curves becomes increasingly complicated as $g$ grows, and each genus above requires new ideas to understand curves of that genus. While many parts of the arguments are specific to a particular genus, there are also similar themes throughout. Some of these themes will be highlighted in our exposition on genus $3$ in the next section.

On the other hand, there are known non-tautological classes in most higher genera.

\begin{thm}[Van Zelm \cite{VZ}, Arena et. al. \cite{cc}] \label{nt}
If $g = 12$ or $g \geq 16$, $\ch^*(\M_g) \neq \mathsf{R}^*(\M_g)$.
\end{thm}

In the cases above, the authors construct an explicit non-tautological cycle. This cycle is very natural --- a fundamental class of locus of curves which are double covers of curves of some smaller fixed genus. We'll explain the case of genus $12$ in the last section.

\section{Tautological generation in genus $3$} \label{s3}
In order to give a sense of the machinery that goes into proving Theorem \ref{t},  we give a proof of the fact that $\ch^*(\M_3) = \mathsf{R}^*(\M_3)$, which was originally established by Faber \cite{F2}. 
While this case is relatively simple, many of the ingredients that have been used in higher genus are already visible and we will highlight these key concepts.
We stress that the method here relies on rational coefficients. It is still an open problem to compute the Chow ring of $\M_3$ with integral coefficients.

The first step is to stratify the moduli space according to different ways we know how to model our curve via an explicit map to projective space. In genus $3$, every curve is hyperelliptic or a smooth plane quartic
\[\M_3 = \{\text{smooth plane quartics}\} \sqcup \{\text{hyperelliptic curves}\} = \P \sqcup \H.\]
Our strategy is to separately obtain generators for the Chow rings of $\P$ and $\H$, and then assemble our information with the excision sequence
\begin{equation} \label{exe} \ch^{i-1}(\H) \to \ch^{i}(\M_3) \to \ch^i(\P) \to 0
\end{equation}
in order to obtain generators for $ \ch^{*}(\M_3)$. 

In higher genus, it has been fruitful to stratify $\M_g$ by \emph{gonality} (the minimal degree of a map $C \to \pp^1$) but sometimes also distinguish other special Brill--Noether loci (loci of curves admitting other specified maps to projective spaces) to deal with separately. 
The decomposition above is the gonality stratification for $\M_3$.
In higher genus, the strategy is similar: we obtain generators for the Chow ring of each piece of a stratification and study their contributions to the Chow ring of $\M_g$ via the excision sequence. A key step is showing that fundamental classes of strata are tautological.

\subsection{The hyperelliptic locus} \label{hyp}
By the Riemann--Hurwitz formula, a hyperelliptic genus $3$ curve $C \to \pp^1$ is a branched over $8$ distinct points. Conversely, any $8$ distinct points on $\pp^1$ determine a hyperelliptic genus $3$ curve. Two configurations of $8$ points determine the same hyperelliptic curve if and only if they are related by an automorphism of $\pp^1$.
It follows that the coarse moduli space of $\H$ is
\[\H \leftrightarrow \M_{0,8}/S_8 = \{\text{$8$ unordered points on $\pp^1$}\}/\PGL_2.\]
We emphasize that $\leftrightarrow$ above is not an equivalence of moduli \emph{stacks}: a generic element of $\H$ has a $\zz/2$-stabilizer, while the generic element of $\M_{0,8}/S_8$ has no stabilizers. However, since they have the same coarse moduli space, $\H$ and $\M_{0,8}/S_8$ have the same rational Chow ring (see Section \ref{comment}).
To determine the Chow ring of $\M_{0,8}/S_8$, consider the pullback along the quotient map
\begin{equation} \label{inj} \ch^*(\M_{0,8}/S_8) \to \ch^*(\M_{0,8})
\end{equation}
The composition of this pullback with the pushforward along the quotient map is multiplication by the degree, which is $|S_8| = 8!$. Hence, it is injective with rational coefficients.

Meanwhile, $\M_{0,8}$ is an open subset of $\aa^5$. Indeed,
when the points are ordered, using automorphisms of $\pp^1$, we can take the first three points to $0, 1, \infty$. The remaining five points lie in the complement of the diagonals in $(\pp^1 \smallsetminus \{0, 1, \infty\})^5 \subset (\aa^1)^5 = \aa^5$. Since $\ch^*(\aa^5) = \qq$, excision implies $\ch^*(\M_{0,8}) = \qq$. In conclusion, as \eqref{inj} is injective, we have
\[\ch^*(\H) = \qq.\]
That is, $\H$ has only its fundamental class in degree $0$ and no other non-trivial classes.

We remark that the \emph{integral} Chow ring of the hyperelliptic locus is computed in \cite{D,EF}, and is non-trivial (it has lots of $2$-torsion, like the Chow ring of $\mathrm{B}(\zz/2)$).

\subsection{The plane quartics}
All non-hyperelliptic genus $3$ curves are realized by their canonical embedding as smooth plane quartics, and conversely every smooth plane quartic is a canonically embedded genus $3$ curve. Let $U \subset H^0(\pp^2, \O_{\pp^2}(4))$ be the open subset of homogeneous degree $4$ polynomials whose vanishing locus in $\pp^2$ is smooth. There is a natural map $U \to \P$ induced by the universal plane quartic over $U$. Equations in the same $\GL_3$ orbit yield isomorphic curves, so we obtain 
\[U/\GL_3 \to \P.\]
It's not hard to see that this map induces an isomorphism of coarse moduli spaces.
Indeed, every isomorphism of non-hyperelliptic genus $3$ curves induces an isomorphism of their canonical models and vice versa, so the $\GL_3$ orbits above are seen to be precisely the isomorphism classes of curves in $\P$.
 \footnote{If we specify that
$\GL_3$ acts by change of coordinates \emph{and scaling by the inverse of the determinant} then this map is in fact an equivalence of stacks and this presentation has been used to compute the integral Chow ring of $\P$ \cite{DFV}. Finding the relations induced by the complement of $U/\GL_3$ inside $H^0(\pp^2, \O_{\pp^2}(4))/\GL_3$ is an example of an ``excision problem," as alluded to in Section \ref{sec:ex}.}

This description as a quotient stack provides us with generators for the Chow ring of $\P$. Indeed we have a diagram like so
\begin{center}
\begin{tikzcd}
\P &   \arrow{l} U/\GL_3 \arrow{r}{\text{open}} & \mathsf{H}^0(\pp^2,\O_{\pp^2}(4))/\GL_3 \arrow{d}{\text{vector bundle}} \\
& & \BGL_3
\end{tikzcd}
\end{center}
We know that pullback along a vector bundle morphism induces an isomorphism on Chow rings and pullback along an open inclusion induces a surjection on Chow rings, so we obtain
\[\qq[c_1,c_2,c_3] = \ch^*(\BGL_3) \cong \ch^*(H^0(\pp^2,\O_{\pp^2}(4))/\GL_3) \twoheadrightarrow \ch^*(\P). \]
What are the pullbacks of $c_1, c_2, c_3$ along this map? By definition, a map $\P \to \BGL_3$ corresponds to a rank $3$ vector bundle on $\P$, and pullbacks of the $c_i$ are the Chern classes of this vector bundle. The vector bundle in question is one whose fiber at each point $C$ is $H^0(K_C)$. In other words, it is the restriction of the Hodge bundle $\mathbb{E}$.\footnote{Or if one wants the $\GL_3$ action to also scale by the inverse of the determinant, the associated vector bundle is $\mathbb{E} \otimes \det \mathbb{E}^\vee$. Either way, its Chern classes are tautological.} Thus, the map above sends $c_i$ to $\lambda_i$. In particular, $\ch^*(\P)$ is generated by the restrictions of tautological classes.

\subsection{Putting the pieces together}
By Section \ref{hyp}, the left-hand side of \eqref{exe} vanishes unless $i = 1$. 
For $i \neq 1$, the fact that $\ch^i(\P)$ is generated by tautological classes implies that $\ch^i(\M_3)$ is too. When $i = 1$, the excision sequence ensures that $\ch^1(\M_3)$ modulo the span of the fundamental class $[\H]$ 
is generated by tautological classes. Thus, to conclude it suffices to show that $[\H]$ is tautological.

\subsubsection{The fundamental class of the hyperelliptic locus}
Determining fundamental classes of subvarieties of $\M_g$ can be challenging in general. However, when a locus can be realized as occurring in the ``expected codimension," we can use degeneracy formulas and they typically show that the class is tautological. We explain this strategy in the case of $\H \subset \M_3$. Similar ideas show that Brill--Noether loci of the expected codimension have tautological fundamental classes, see \cite{F1}.

What distinguishes hyperelliptic curves? They are the curves that have two points $p_1, p_2$ such that $h^0(\O_C(p_1 + p_2)) = 2$. Equivalently, by Riemann--Roch, there exists a pair of points which fail to impose independent conditions on the canonical: $h^0(\omega_C(-p_1 - p_2)) = g-1 = 2$.
Recycling notation from Section \ref{tr}, let $\pi: \C^2 \to \M_3$ be the fiber product of two copies of the universal curve and consider the vector bundle $\Omega_2$ whose fiber over $[C, p_1, p_2]$ is $H^0(\omega_C|_{p_1 + p_2})$. As before, the evaluation map
\[H^0(\omega_C) \to H^0(\omega_C|_{p_1 + p_2})\]
globalizes to a map of vector bundles
\[\phi: \pi^* \mathbb{E} \to \Omega_2.\]

In the fiber over most points $[C, p_1, p_2]$, this map from a rank $3$ vector bundle to a rank $2$ vector bundle has a $1$-dimensional kernel, but it has a $2$-dimensional kernel precisely when $h^0(\omega_C(-p_1 - p_2)) = 2$.
That is, 
\[Z := \{[C, p_1, p_2] : h^0(\O_C(p_1 + p_2)) = 2\} \subset \C^2\]
is the locus where $\phi$ drops rank. The expected codimension of this degeneracy locus is $(\dim \ker)(\dim \coker) = 2\cdot 1$. 
We claim that $Z$ is in fact codimension $2$ in $\C^2$.
Indeed, $\H$ is codimension $1$ in $\M_3$ and the fibers of $Z \to \H$ have dimension $1$ (all choices of pairs of points in the same fiber of the hyperelliptic map). Thus, 
\[\dim Z = \dim \H + 1 = \dim \M_3 - 1 + 1 = \dim \C^2 - 2.\]
Since $Z$ has the expected codimension, its class is given by the \emph{Porteous formula}. This is a formula for the fundamental class of a degeneracy locus in terms of the Chern classes of the source and target vector bundles, see \cite[Chapter 12]{3264}.
In particular,
$[Z]$ is a polynomial in $K_1, K_2, D_{1,2}, \pi^*\kappa_j, \pi^*\lambda_i$.

Now, $Z \to \H$ has $1$-dimensional fibers, so if we push forward $[Z]$ along $\pi$, we will just get zero. However, if we intersect with a class on $\C^2$ before pushing forward, we may cut down these fibers and hope to obtain a multiple of $[\H]$. For example, if we intersect $Z$ with $D_{1,2}$, this is the locus of $[C, p, p]$ such that $h^0(\O_C(2p)) = 2$, i.e. $p$ is a branch point of the hyperelliptic map. There are $8$ branch points, so $Z \cap D_{1,2} \to \H$ is finite of degree $8$. Thus,
\begin{align*} 
8[\H] &= \pi_*[Z \cap D_{1,2}] = \pi_*([Z] \cdot D_{1,2})\\
&= \pi_*(\text{polynomial in  $K_1, K_2, D_{1,2}, \pi^*\kappa_j, \pi^*\lambda_i$}) \in \mathsf{R}^1(\M_3),
\end{align*}
showing that $[\H]$ is tautological as desired.

\section{A non-tautological class in genus $12$}
In the last section, we explain the construction of a non-tautological class in the Chow ring of $\M_{12}$.
Although we just treat genus $12$, the non-tautological classes discovered in higher genus are obtained by similar ideas. This technique for finding non-tautological classes was initiated in \cite{GP}, built upon in \cite{VZ} (which includes the case of genus $12$ presented here), and generalized further in \cite{cc}. The argument uses properties of the tautological ring of the compactificaiton $\Mb_{g,n}$, parameterizing stable genus $g$ curves with $n$ marked points.  We refer the reader to \cite[Chapter X]{ACG} for definitions and background on stable curves.

\subsection{The tautological ring of $\Mb_{g,n}$} \label{bart}
Like the Grassmannian in Section \ref{gkn}, the moduli space of pointed stable curves $\Mb_{g,n}$ admits a stratification whose strata are labeled by nice combinatorial objects called \emph{stable graphs}. 
In this graph, each component corresponds to a vertex, labeled by the genus of its normalization; each node corresponds to an edge; and each marked point is represented by a half-edge. For example the curve on the left below has dual graph pictured on the right.
\begin{equation}
\includegraphics[width=3in]{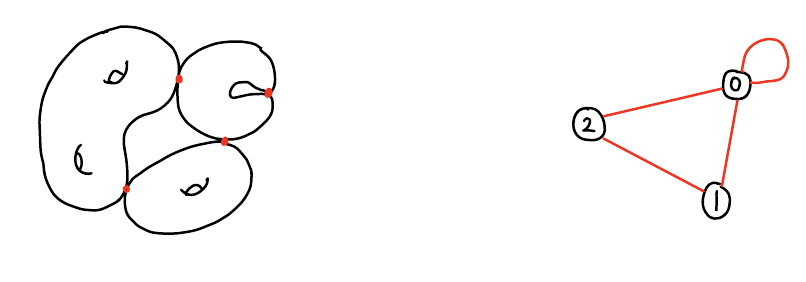}
\end{equation}
The stability condition is the requirement that every genus $0$ vertex has at least three half edges coming out of it and every genus $1$ vertex has at least one half edge. To see the genus of a nodal curve, imagine smoothing out the nodes. For example, the curve pictured above has genus $5$.
The closure of the stratum associated to a dual graph $\Gamma$ is the union of locally closed strata labeled by graphs $\Gamma'$ which admit an edge contraction $\Gamma' \to \Gamma$.
This closed stratum is the
image of a gluing map
\[\xi_{\Gamma} \colon \prod_{v \in \Gamma} \Mb_{g_v,n_v} \to \Mb_{g,n}.\]
Of particular interest for us in genus $12$ is the gluing map
\begin{equation} \label{glue1} i\colon \Mb_{1,11} \times \Mb_{1,11} \to \Mb_{12}
\end{equation}
whose image is nodal curves as in the picture below (and their specializations):
\begin{equation} \label{glue2}
\includegraphics[width=3in]{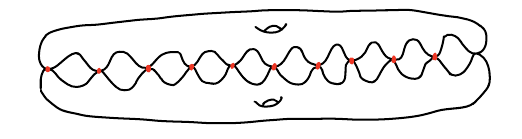}
\end{equation}

The other natural maps between moduli spaces of stable curves are the forgetful maps $f \colon \Mb_{g,n+1} \to \Mb_{g,n}$, defined by forgetting the last marking and stabilizing (a process that contracts any genus $0$ components with fewer than $3$ half-edges). In fact, $f \colon \Mb_{g,n+1} \to \Mb_{g,n}$ is the universal curve over $\Mb_{g,n}$ and comes equipped with $n$ disjoint sections 
\[\sigma_1, \ldots, \sigma_n\colon \Mb_{g,n} \to \Mb_{g,n+1}\]
corresponding to the $n$ markings.

\begin{definition} \label{bar-taut}
The tautological rings $\mathsf{R}^*(\Mb_{g,n}) \subseteq \ch^*(\Mb_{g,n})$ are the smallest system of subrings containing $1 \in \ch^0(\Mb_{g,n})$ and closed under pushforward along the gluing and forgetful maps. 
\end{definition}

From the above definition, it is not obvious how to find generators for the tautological ring $\mathsf{R}^*(\Mb_{g,n})$, or even that this ring is finitely generated. However, there is a nice list of generators for the tautological ring called \emph{decorated boundary strata}.
To define them, we first define
\[\psi_i = c_1(\sigma_i^* \omega_f) \in \ch^1(\Mb_{g,n}) \qquad \text{and} \qquad \kappa_j = f_*(\psi_{n+1}^{j+1}) \in \ch^j(\Mb_{g,n}).\]
To see that the classes $\psi_i$ (and hence also $\kappa_j$) are tautological according to Definition \ref{bar-taut}, we realize the section $\sigma_i$ as a gluing map
\[\sigma_i \colon \Mb_{g,n} = \Mb_{g,\{1,\ldots, \hat{i}, \ldots ,n\}} \times \Mb_{0,\{p',i,n+1\}} \to \Mb_{g,n+1}\]
which glues $p$ to $p'$. By Definition \ref{bar-taut}, the class $f_*(\sigma_{i*}1)^2$ must be tautological. On the other hand, using the self-intersection formula (see Section \ref{si}), we have
\[(\sigma_{i*}1)^2 = \sigma_{i*}(c_1(N_{\sigma_i(\Mb_{g,n})/\Mb_{g,n+1})}) = \sigma_{i*}(c_1(\sigma_i^*\omega_f^\vee)) = \sigma_{i*}(-\psi_i).\]
The second equality above follows from identifying the normal bundle to the section with the restriction of the relative tangent bundle to the section (similarly to as in \eqref{normal}). Finally, since $\sigma_i$ is a section of $f$, we have
\[ \mathsf{R}^1(\Mb_{g,n}) \ni f_*(\sigma_{i*}1)^2 = f_*\sigma_{i*}(-\psi_i) = -\psi_i,\]
showing that $\psi_i$ is tautological.

Now, given any gluing map 
$\xi_{\Gamma}\colon \prod_{v \in \Gamma} \Mb_{g_v,n_v} \to \Mb_{g,n}$, classes of the following form are tautological by Definition \ref{bar-taut}:
\begin{equation} \label{decst} \xi_{\Gamma*
}\left( \prod_{v \in \Gamma} \mathrm{pr}_v^*(\text{monomial in $\psi_i$ and $\kappa_j$})\right).
\end{equation}
A class of the form \eqref{decst} is called a \emph{decorated boundary stratum}. This is because $\xi_{\Gamma*}(1)$ is the fundamental class of a boundary stratum. We also need to allow $\kappa$ and $\psi$ ``decorations" in order to get a system of classes closed under multiplication. The graph $\Gamma$ is allowed to be the trivial graph with one vertex of genus $g$, which allows one to obtain the usual $\psi$ and $\kappa$ classes. In particular, there is a restriction map $\mathsf{R}^*(\Mb_{g,n}) \to \mathsf{R}^*(\M_{g,n})$ which is surjective.

Using excess intersection, one can give a formula for the product of two decorated boundary strata as a linear combination of decorated boundary strata \cite[Appendix A]{GP}.
 One should think of 
 decorated strata as analogues of Schubert cycles, and this rule for multiplication as the 
analogue of Schubert calculus (see Section \ref{gkn}).

Going further, \cite[Appendix A]{GP} actually determines a formula for the pullback of decorated boundary strata under pullbacks along gluing maps. The key fact for us is that these pullbacks of tautological classes are tautological in the following sense.

\begin{lem}[Appendix A of \cite{GP}] \label{pb}
Suppose $\xi_\Gamma \colon \prod_{v \in \Gamma} \Mb_{g_v,n_v} \to \Mb_{g,n}$  is a gluing map and $v \in \mathsf{R}^*(\Mb_{g,n})$. Then, $\xi_\Gamma^*z$ lies in the image of $\bigotimes_{v \in \Gamma} R^*(\Mb_{g_v,n_v}) \to \mathsf{A}^*(\prod_{v \in \Gamma} \Mb_{g_v,n_v})$. In particular, the K\"unneth components of the image of $\xi_\Gamma^*v$ under the cycle class map are tautological.
\end{lem}

\subsection{Non-tautological classes}
Lemma \ref{pb} suggests a method for proving a class $v \in \ch^*(\Mb_{g,n})$ is non-tautological: if $\xi_\Gamma^*v$ has a non-tautological K\"unneth component, then $v$ cannot be tautological. Van Zelm executed this strategy when $\xi_\Gamma$ is the gluing map $i$ shown in \eqref{glue1}, and $v$ the fundamental class of the bielliptic locus. A curve is \emph{bielliptic} if it is a double cover of a genus $1$ curve. Let 
\[\mathcal{B}_g = \{[C] \in \M_g :  \exists \ C \to E \text{ of degree 2}\} \]
denote the locus of bielliptic curves.
By the Riemann--Hurwitz formula, a genus $12$ double cover of a genus $1$ curve has
$22$
branch points. The moduli of such a double cover is specified by a genus $1$ curve with $22$ points, 
 which is a $22$-dimensional moduli space. Furthermore, as the following lemma shows, any degree $2$ map from a genus $12$ curve to a genus $1$ curve is unique.

 \begin{lem}
 A genus $12$ curve cannot admit two distinct degree $2$ maps to genus $1$ curves.
 \end{lem}
 \begin{proof}
 Suppose for contradiction that $C$ is a genus $12$ curve and $C \to E$ and $C \to E'$ are distinct double covers of genus $1$ curves. Taking their product yields a map $C \to E \times E'$. Since the two maps are not equal, $C$ is birational onto its image in $E \times E'$. The canonical bundle of the image curve $\overline{C}$ can be computed using adjunction. Recalling that the canonical bundle of $E$ and $E'$ are trivial, we find that $K_{\overline{C}} = \O_{E \times E'}(\overline{C})|_{\overline{C}}$. Since $C \to E$ and $C \to E'$ are each degree $2$, the class of $\overline{C}$ is numerically equivalent $2F + 2F'$ where $F$ and $F'$ are the fibers of the projection maps from $E \times E'$ to $E$ and $E'$. (Here, \emph{numerical equivalence} is defined similarly to rational equivalence, but where the base curve of the family can have any genus. Importantly, dimension zero numerical equivalence classes have a well-defined degree.) Since $\deg F^2 = \deg F'^2 = 0$ and $\deg F\cdot F' = 1$, we find $\deg K_{\overline{C}} = (2F + 2F')^2 = 8$. But then the arithmetic genus of $\overline{C}$ would be $5$. However, the arithmetic genus of $\overline{C}$ is at least the genus of $C$, yielding a contradiction.
 \end{proof}

We conclude that the locus of bielliptic curves has dimension $22$. 
Since $\dim \M_{12} = 33$, we see that the codimension of $\B_{12} \subset \M_{12}$ is $11$. More generally, an analogous argument shows that the codimension of $\B_{g} \subset \M_g$ is $g - 1$.

As a first step, we shall show that the fundamental class of the closure of the bielliptic locus $[\overline{\B}_{12}] \in \ch^{11}(\Mb_{12})$ is not tautological.
One can understand the curves lying in $\overline{\B}_{12}$ using the theory of \emph{admissible covers}. 
It is not hard to see that
$i^{-1}(\overline{\B}_{12})$ contains the diagonal $\Delta \subset \Mb_{1,11} \times \Mb_{1,11}$. Indeed, when the same $11$-pointed elliptic curve is glued to itself, the resulting curve admits an involution (picture swapping the top and bottom halves of \eqref{glue2}) whose quotient is that genus $1$ curve. The key lemma established by Van Zelm is that $i^{-1}(\overline{\B}_{12}) = \Delta \cup Z$ where 
\[Z \subset \partial \M_{1,11} \times \Mb_{1,11} \cup \Mb_{1,11} \times \partial \M_{1,11}.\]
Thus, for some non-zero $\alpha$ (the multiplicity of the diagonal in the preimage), we have
\begin{equation} \label{xa} i^*[\overline{\B}_{12}] = \alpha[\Delta] + j_*z
\end{equation}
where $j:  \partial \M_{1,11} \times \Mb_{1,11} \cup\Mb_{1,11} \times \partial \M_{1,11}  \to \Mb_{1,11} \times \Mb_{1,11}$ is the inclusion of the boundary and $z$ is some class in the Chow groups of the boundary.

Recall that the cycle class map $\ch^*(X) \to \mathsf{H}^*(X)$ has image that lies in even degree.
In particular, any odd degree cohomology class on $\Mb_{g,n}$ is necessarily non-tautological. It turns out that the first non-zero odd degree cohomology group on one of these moduli spaces is $\mathsf{H}^{11}(\Mb_{1,11})$. (This really is the first non-vanishing cohomology group, as $\mathsf{H}^k(\Mb_{g,n}) = 0$ for odd $k \leq 9$ and all $g$ and $n$ \cite{ArbarelloCornalba,BFP}, and $\mathsf{H}^{11}(\Mb_{g,n}) = 0$ unless $g = 1$ and $n \geq 11$ \cite{h11}.) 

Assuming this fact and \eqref{xa}, we sketch a proof of the following.

\begin{lem} \label{abar}
The K\"unneth components of $i^*[\overline{\B}_{12}]$ are non-tautological. Hence, $[\overline{\B}_{12}]$ is non-tautological.
\end{lem}
\begin{proof}
Given a basis $\{e_i\}_{i\in I}$ for $\mathsf{H}^*(\Mb_{1,11})$ with dual basis $\{\hat{e}_i\}_{i\in I}$ the cohomology class of the diagonal can be
written as
\[[\Delta] = \sum_{i \in I}(-1)^{\deg e_i} e_i \otimes \hat{e}_i.\]
Since  $\mathsf{H}^{11}(\Mb_{1,11}) \neq 0$, we see that $[\Delta]$ has a non-tautological K\"unneth component.
On the other hand, we show that all classes supported on the boundary of $\Mb_{1,11}$ are tautological. 
Let $\M_\Gamma \subset \partial \M_{1,11}$ be the locally closed stratum of the boundary consisting of curves whose dual graph is $\Gamma$. It suffices to show that $\ch^*(\M_\Gamma) = \qq$, since then the excision sequence will show that $\ch_*(\partial \M_{1,11})$ is generated by the fundamental classes of closures of boundary strata, all of which are tautological.

To prove the claim that $\ch^*(\M_\Gamma) = \qq$, consider the
finite surjective map
\[\M_{1,n} \times \prod \M_{0,n_i} \to \M_\Gamma\]
(or with the first factor omitted if $\Gamma$ is a graph with a single loop).
Now, each $\M_{0,n_i}$ is an open subset of affine space, so $\prod \M_{0,n_i}$ is an open subset of some affine space, say $\mathbb{A}^N$. Moreover, Belorousski \cite{Bel} proves, using techniques of a similar flavor to Section \ref{s3}, that $\ch^*(\M_{1,n}) = \qq$ for $n \leq 10$. 
Putting these facts together
we see that
\[\qq = \ch^*(\M_{1,n}) = \ch^*(\M_{1,n} \times \aa^N) \twoheadrightarrow \ch^*(\M_{1,n} \times \prod \M_{0,n_i}) \twoheadrightarrow \ch^*(\M_\Gamma). \]
Above, the first surjective arrow is the pullback along an open inclusion; the second is a pushforward along a surjective morphism with rational coefficients.

Having established the claim, we see that each K\"unneth component of $j_*z$ has at least one tautological tensor factor. In particular, $j_*z$ has no K\"unneth component in $\mathsf{H}^{11} \otimes \mathsf{H}^{11}$.
Thus, we conclude that the sum $\alpha [\Delta] + j_*z$ has a non-tautological K\"unneth component, completing the proof.
\end{proof}

To get from the fact that $[\overline{\B}_{12}]$ is non-tautological to the fact that $[\B_{12}]$ is non-tautological requires additional work since, a priori, there could be a non-tautological class on the boundary of $\Mb_{12}$ whose difference with $[\overline{\B}_{12}]$ is tautological.
In fact, although Van Zelm established that $[\overline{\B}_{g}] \notin \mathsf{R}^*(\Mb_g)$ for $g \geq 12$, it still remains an open problem whether or not $[\B_g]$ (or indeed any algebraic cycle) is non-tautological on $\M_g$ for $g = 13, 14, 15$.

The key to obtaining results about non-tautological classes on the \emph{interior} $\M_g$ is to tensor cohomology up to $\mathbb{C}$-coefficients and look at also at the Hodge type of the different K\"unneth components of $i^*[\overline{\B}_{12}]$. 
It turns out that, not only is $\mathsf{H}^{11}(\Mb_{1,11}) \neq 0$, but $\mathsf{H}^{11,0}(\Mb_{1,11}) \neq 0$.
This fact can also be used to simplify the argument that $[\overline{\B}_{12}]$ is non-tautological in Lemma \ref{abar}, as we now explain. With this additional information, one does not require a such thorough understanding of classes supported on the boundary of $\Mb_{1,11}$. This was one of the key insights of \cite{cc}.
The relevant fact is that a class pushed forward from proper subvariety will always have Hodge type $(p,q)$ with $p, q \geq 1$.
Thus, even without knowing that classes supported on $\partial \M_{1,11}$ are tautological, one can see that the class $j_*z$ will not be able to cancel the Hodge--K\"unneth component of $\alpha [\Delta]$ of type $\mathsf{H}^{11,0} \otimes \mathsf{H}^{0,11}$.

We now explain a proof of the following proposition using ideas in \cite{cc}. We stick to the case of the bielliptic locus in genus $12$ for clarity, but as mentioned earlier, these ideas generalize substantially, see \cite[Theorem B]{cc} for a general statement. Combined with \cite{FT}, there are now known non-tautological algebraic cycles on $\M_{g,n}$ for all but finitely many pairs $(g, n)$.

\begin{prop}
The class $[\B_{12}]$ is non-tautological.
\end{prop}
\begin{proof}
Suppose for contradiction that $[\B_{12}]$ is tautological. Then, by excision,
there exists an identity
\[[\overline{\B}_{12}] = T +  B \in \ch^{11}(\Mb_{12}),\]
where $T \in \mathsf{R}^{11}(\Mb_{12})$ and $B$ is a class supported on the boundary. Applying pullback along the gluing map \eqref{glue1} to this formula, and also using \eqref{xa} we find
\begin{equation} \label{theone} \alpha [\Delta] + j_*z = i^*[\overline{\B}_{12}] = i^* T + i^*B.
\end{equation}
We now consider the images of the above classes under the cycle class map.

Since $\mathsf{H}^{11,0}(\Mb_{1,11}) \neq 0$, we have that $[\Delta]$ has a non-zero Hodge--K\"unneth component of type $\mathsf{H}^{11,0} \otimes \mathsf{H}^{0,11}$.
Since $j_*z$ is pushed forward from the boundary, it has no Hodge--K\"unneth component of type $\mathsf{H}^{11,0} \otimes \mathsf{H}^{0,11}$. 
Hence, the left-hand side of \eqref{theone} has a non-zero Hodge--K\"unneth component of type $\mathsf{H}^{11,0} \otimes \mathsf{H}^{0,11}$. Meanwhile, since it is tautological, $i^*T$ has no Hodge--K\"unneth component of type $\mathsf{H}^{11,0} \otimes \mathsf{H}^{0,11}$. Finally, we claim that $i^*B$ has no 
Hodge--K\"unneth component of type $\mathsf{H}^{11,0} \otimes \mathsf{H}^{0,11}$.  This will produce a contradiction and complete the proof.

By writing $B$ as a linear combination of fundamental classes of subvarieties pushed forward from boundary divisors, to prove the claim, it suffices to show that $i^*\xi_*Z$ has no 
Hodge--K\"unneth component of type $\mathsf{H}^{11,0} \otimes \mathsf{H}^{0,11}$ where $\xi$ is one of the gluing maps
\[\xi_{\mathrm{irr}}\colon \Mb_{11,2} \to \Mb_{12} \qquad \text{or} \qquad \xi_{a}\colon \Mb_{a,1} \times \Mb_{12-a,1} \to \Mb_{12},\]
and $Z$ is the fundamental class of an irreducible subvariety.

To treat the first case, notice that the image of $\xi_{\mathrm{irr}}$ contains the image of $i$; indeed, the stable graph associated to $i$ edge contracts to the single loop graph associated to $\xi_{\mathrm{irr}}$. In particular, we may write $i = \xi_{\mathrm{irr}} \circ i'$ where $i':\Mb_{1,11} \times \Mb_{1,11} \to \Mb_{11,2}$ glues the first ten pairs of points, leaving two marked points. Then, using the excess intersection formula, one finds
\[i^*\xi_{\mathrm{irr}*}Z = i'^* \xi_{\mathrm{irr}}^*\xi_{\mathrm{irr}*} Z = i'^*(-\psi_1 - \psi_2) Z.\]
The important fact is that $\psi_j \in \mathsf{H}^{1,1}(\Mb_{11,2})$ and furthermore, $i'^*\psi_j$ is either $\psi_{11} \otimes 1$ or $1 \otimes \psi_{11}$, and so lives in
$\mathsf{H}^{1,1} \otimes \mathsf{H}^{0,0}$ or $\mathsf{H}^{0,0} \otimes \mathsf{H}^{1,1}$.
In particular, the right-hand side of the above equation has no Hodge--K\"unneth component of type $\mathsf{H}^{11,0} \otimes \mathsf{H}^{0,11}$.

For the second case, notice that the image of $\xi_a$ does not contain the image of $i$; indeed, the stable graph associated to $i$ admits no edge contraction to the graph associated to $\xi_a$, which has a separating edge. It follows that $i^*\xi_{a*} Z$ is supported on the boundary of $\Mb_{1,11} \times \Mb_{1,11}$. From this, one sees that it also has no Hodge--K\"unneth component of type $\mathsf{H}^{11,0} \otimes \mathsf{H}^{0,11}$.
\end{proof}

\end{document}